\documentclass{amsart}
\usepackage{mathtools}
\usepackage{amsfonts}
\usepackage{amsmath}
\usepackage{amsthm}
\usepackage[square, numbers]{natbib}
\usepackage{cleveref}
\usepackage{amssymb}
\usepackage{tikz-cd}
\usepackage{tikz}
\usepackage{color}
\usepackage{array}
\usepackage{multirow}
\usepackage{gensymb}
\usepackage{tabularx}
\usepackage{graphicx}
\graphicspath{ {images/} }
\usepackage{caption}
\usepackage{enumerate}
\usepackage{setspace}
\usepackage{float}
\usepackage{fullpage}
\title{Euler Systems and Selmer Bounds for GU(2,1)}
\author{Muhammad Manji}

\newtheorem{thm}{Theorem}[section]
\newtheorem{lem}[thm]{Lemma}
\newtheorem{prop}[thm]{Proposition}
\newtheorem{cor}[thm]{Corollary}
\newtheorem{conj}[thm]{Conjecture}
\theoremstyle{definition}
\newtheorem{defn}[thm]{Definition}
\theoremstyle{remark}
\newtheorem{rmk}[thm]{Remark}
\begin{document}
	\vfill
	\begin{abstract}
		We investigate properties of the Euler system associated to certain automorphic representations of the unitary similitude group GU(2,1) with respect to an imaginary quadratic field $E$, constructed by Loeffler-Skinner-Zerbes. By adapting Mazur and Rubin's Euler system machinery we prove one divisibility of the ``rank 1" Iwasawa main conjecture under some mild hypotheses. When $p$ is split in $E$ we also prove a ``rank 0" statement of the main conjecture, bounding a particular Selmer group in terms of a $p$-adic distribution conjecturally interpolating complex $L$-values. We then prove descended versions of these results, at integral level, where we bound certain Bloch--Kato Selmer groups. We will also discuss the case where $p$ is inert, which is a work in progress.
	\end{abstract}
	\maketitle
	\vspace{30pt}
	\tableofcontents
	\section{Introduction}
	\par In the past decade there has been a series of developments in the theory of Euler systems, in particular the construction of new Euler systems for Shimura varieties attached to linear algebraic groups such as $\mathrm{GL}_2 \times_{\mathrm{GL}_1} \mathrm{GL}_2$, $\mathrm{GSp}_4$, as well as the construction for GU(2,1) in \citep{LSZ21} discussed in this paper. In the late 20th and early 21st century, examples of Euler systems were hard to find but a general machinery was developed by Kolyvagin, Mazur, Rubin and others to use Euler systems to bound Selmer groups in terms of special values of $p$-adic $L$-functions. One main aim of this is to solve the Bloch--Kato conjecture for these cases. With an increasing array of examples, new techniques are being developed to extend existing results to new cases. 
	\par The aim of this paper is to look at a certain part of the machinery, the Iwasawa main conjectures, and adapt the results for the GU(2,1) Euler system. Partial results have been obtained (see Theorem 3.3 and corollary 3.6), and will hopefully lead to progress towards the Bloch--Kato conjecture for GU(2,1). 
	\par In \S1 we discuss automorphic representations for GU(2,1) and the existence of associated Galois representations coming from the Langlands program. We also discuss ramification, Hodge--Tate weights and ordinarity conditions which will be important for later sections. 
	\par We build on this in \S2 by looking at certain subrepresentations of our Galois representation $V$ after twisting by Hecke characters. We compute all such subrepresentations and present the information in a diagram, which conjecturally acts as a roadmap telling us where we can find Euler systems of different ranks associated to $V$. Moreover, we present the existing construction of a rank 1 Euler system for $V$ from \citep{LSZ21}, which matches the conjectural picture. We make a clear distinction between the cases where $p$ is split and inert in $E$. The latter case will present more challenges as we progress.
	\par In \S3 we adapt versions of Mazur and Rubin's Euler system machinery from \citep{MR04} in order to prove one divisibility of the rank 1 Iwasawa main conjecture for such $V$ (i.e.\ a statement in terms of Selmer groups, without $p$-adic $L$-functions), and reduce the proof to a lemma about vanishing of some local $H^0$ groups. We will also use the language of Selmer complexes and derived base change to establish a version of the main conjecture at an integral level rather than over the Iwasawa algebra.
	\par In \S4 we will deal only with the case where $p$ splits in $E$, and use a Coleman map to identify the Euler system with a distribution in the Iwasawa algebra which we call a ``motivic $p$-adic $L$-function", which conjecturally interpolates complex $L$-values. Using the results from \S3, we will bound certain Selmer groups using this distribution, and use this result to gain information about the Bloch--Kato Selmer group. In the case where $p$ is inert new methods are needed, we outline a strategy which we will use to attack the problem (which is a work in progress). In order to tie the distribution to an analytic $p$-adic $L$-function and prove the full statements of the Iwasawa main conjecture we would need an explicit reciprocity law, and the progress of other authors towards this is briefly discussed at the end of the section.

	\section{Setup}
	\par Suppose $E$ is an imaginary quadratic field of discriminant $-D$, with non-trivial automorphism $c: x \mapsto \bar{x}$. Identify $E \otimes \mathbb{R} \cong \mathbb{C}$ such that $\delta= \sqrt{-D}$ has positive imaginary part. Let $J \in \mathrm{GL}_3$ be the hermitian matrix
	\begin{equation*} J = \begin{pmatrix}
		0 & 0 & \delta^{-1} \\
		0 & 1 & 0\\
		-\delta^{-1} & 0 & 0 
		\end{pmatrix}.
	\end{equation*}
	Let $G$ be the group scheme over $\mathbb{Z}$ such that for a $\mathbb{Z}$-algebra $R$,
	$$ G(R) = \{ (g, \nu) \in \mathrm{GL}_3(\mathcal{O}_E \otimes R) \times R^\times : \prescript{t}{}{\bar{g}}Jg=\nu J \}.$$
	Then $G(\mathbb{R})$ is the unitary similitude group GU(2,1), which is reductive over $\mathbb{Z}_l$ for all $l \nmid D$. We are interested in cuspidal automorphic representations of $G$, and we can relate ($L$-packets of) these to ($L$-packets of) cuspidal automorphic representations of $\mathrm{Res}^{E}_{\mathbb{Q}}(\mathrm{GL}_3 \times \mathrm{GL}_1)$ via base change, following the notation of \citep{LSZ21}, via the referenced result of Mok.
	
	\begin{defn} A regular algebraic, essentially conjugate self-dual, cuspidal (``RAECSDC") representation of $\mathrm{GL_3}/E$ is a pair $(\Pi, \omega)$ such that $\Pi$ is a cuspidal automorphic representation of $\mathrm{GL}_3 $ and $\omega$ is a Hecke character over $\mathbb{Q}$ such that $\Pi_\infty$ is regular algebraic and $\Pi^c = \Pi^\vee \otimes (\omega \circ N_{E/\mathbb{Q}})$. Say $\Pi$ is RAECSDC if there exists $\omega$ such that $(\Pi, \omega)$ is RAECSDC. \end{defn}
	
	\begin{thm}[\citep{Mok15} and others] Given a RAECSDC representation $(\Pi, \omega)$ of $\mathrm{GL}_3/E$, there exists a unique globally generic, cuspidal automorphic representation $\pi$ of $G$ such that the base change of $\pi_v$ at a place $w \mid v$ is $\Pi_w$, and $\pi$ has central character $\chi^{c}_{\Pi}/(\omega \circ N_{E/\mathbb{Q}})$. Moreover $\pi$ is essentially tempered for all places $v$, and $\pi_\infty$ is regular algebraic for $G(\mathbb{R})$. \end{thm}
	\begin{defn} Say a regular algebraic representation $\pi$ of $G(\mathbb{A})$ is non-endoscopic if it arises from a RAECSDC representaion $(\Pi, \omega)$ via the above construction. \end{defn}
	
	\par Given such an automorphic representation, it is natural to ask whether we can associate to it a Galois representation in line with the Langlands program. The next theorem shows the answer is yes, and after this we can start to think about associated Euler systems. 
	\par Let $(\Pi, \omega)$ be a RAECSDC representation of $\mathrm{GL}_3/E$, and let $w$ be a prime for which $\Pi_w$ is unramified. Let $q=\mathrm{Nm}(w)$ and define $P_w(\Pi,X)$ to be the Hecke polynomial i.e.\ satisfying $$ P_w(\Pi, q^{-s})^{-1} = L(\Pi_w, s)$$ .
	
	\begin{thm}[\citep{BLGHT11}, Theorem 1.2]
	The coefficients of $P_w(\Pi,X)$ lie in an extension $F_{\Pi}$ of $E$ independent of $w$, and for each place $\mathfrak{P} \mid p$ of $F_{\Pi}$, there is a 3-dimensional $\mathrm{Gal}(\overline{E}/E)$ representation $V_{\mathfrak{P}}(\Pi)$ over $F_{\Pi,\mathfrak{P}}$, uniquely determined up to semisimplification, such that for $w \nmid p$ such that $\Pi_w$ is unramified,
	$$ \mathrm{det}(1-X\cdot\mathrm{Frob}_w^{-1}:V_{\mathfrak{P}}(\Pi)) = P_w(\Pi, qX)$$ where $\mathrm{Frob}_w$ is the arithmetic Frobenius. \end{thm}

	\noindent The rest of $\S 1$ will establish a few properties of this representation which we will need later. 
	\begin{prop}[\citep{Xia19}] Fixing $\Pi$ and letting $p$ vary, the set of rational primes $p$ such that $V_{\mathfrak{P}}(\Pi)$ is irreducible for all $\mathfrak{P} \mid p$ has density 1. \end{prop}
	\noindent This justifies an assumption we use from now on, that $V_{\mathfrak{P}}(\Pi)$ is irreducible.
	
	\par Recall $\Pi$ is regular algebraic, so we can define its weight at each embedding $\tau: E \hookrightarrow F_{\Pi}$ which is a triple of integers $a_{\tau,1} \geq a_{\tau,2} \geq a_{\tau,3}$.  Twisting by a Hecke character if necessary, we can assume that $\Pi$ has weight $(a+b,b,0)$ at the identity embedding and $(a+b,a,0)$ at the conjugate embedding for some $a,b \geq 0$. When discussing Galois representations, for convenience we will be use the convention of Hodge numbers (negative of Hodge--Tate weights), so the cyclotomic character will have Hodge number $-1$.
	
	\begin{prop}[\citep{BLGHT11}] The representation $V_{\mathfrak{P}}(\Pi)$ is de Rham at all primes above $p$, and has Hodge numbers $\{ 0, 1+b, 2+a+b \}$ at the identity embedding $E \hookrightarrow F_{\Pi,\mathfrak{P}}$ and $\{ 0, 1+a, 2+a+b \}$ at the conjugate embedding. Moreover the coefficients of $P_w(\Pi,qX)$ lie in $\mathcal{O}_{F_\Pi}$ \end{prop}

	\par From here on we will assume $\Pi_v$ is unramified for each place $v \mid p$, so the Hecke polynomials are defined. We want to make further assumptions about the local behaviour of $\Pi$, which we will do in the next few sections. 
	\begin{defn} Say $\Pi$ is ordinary at $v \mid p$ (with respect to $\mathfrak{P} \mid v$ of $F_\Pi$) if $P_v(\Pi,qX)$ has a factor $(1-\alpha_v X)$ with $\mathrm{val}_{\mathfrak{P}}(\alpha_v)=0$. \end{defn}
	\par Following arguments in \citep[\S2]{BLGHT11}, using $p$-adic Hodge theory one can show that $\Pi$ is ordinary at $v$ if and only if the dual representation $V_{\mathfrak{P}}(\Pi)^*$ has a codimension 1  Galois-invariant subspace $\mathcal{F}^1_v V_{\mathfrak{P}}(\Pi)^*$ at $v$ such that $V_{\mathfrak{P}}(\Pi)^*/\mathcal{F}^1_v$ is unramified and $\mathrm{Frob}_v$ acts on the quotient by $\alpha_v$. Since $\Pi$ is conjugate self-dual (up to a twist), we can see that such a subspace exists if and only if $V_{\mathfrak{P}}(\Pi)$ has a codimension 2 invariant subspace at $v$ with compatible action of $\mathrm{Frob}_{\bar{v}}$. If $\Pi$ is ordinary at all primes above $p$ then $V_{\mathfrak{P}}(\Pi)$ and its dual preserve a full flag of invariant subspaces at each prime above $p$, and we will use the notation $\mathcal{F}^i_v$ for both:
	$$ 0 \subset \mathcal{F}^2_v V_{\mathfrak{P}}(\Pi) \subset  \mathcal{F}^1_v V_{\mathfrak{P}}(\Pi) \subset V_{\mathfrak{P}}(\Pi). $$
	We will henceforth assume ordinarity at all primes above $p$ as we will need the full flag. Relaxed ordinarity conditions may be considered in future work.
	
	\begin{rmk} We have an equivalent statement of ordinarity, which will be useful later. If $V$ is ordinary, there exist crystalline characters $\chi_{i,v}$ such that 
	\begin{equation*} V|_{\mathrm{Gal}(\overline{E_v}/E_v)} \sim \begin{pmatrix}
			\chi_{1,v} & * & * \\
			0 & \chi_{2,v} & * \\
			0 & 0 & \chi_{3,v} \\
		\end{pmatrix}.
	\end{equation*}
	The crystalline characters are graded pieces of $V|_{\mathrm{Gal}(\overline{E_v}/E_v)}$ with Hodge--Tate weights in decreasing order as $i$ increases, and this will help us understand how the crystalline Frobenius map $\phi$ acts on the Dieudonné module $\mathbb{D}_{\mathrm{cris}}(V)$. Later we will use known classifications of local crystalline characters to obtain information about cohomology groups of $V$. \end{rmk}

	\section{Local conditions}  	
	Let $V$ be a $p$-adic Galois representation over a number field $K$ which is Hodge--Tate at the primes above $p$.
	\begin{defn} Let $v$ be a prime above $p$. A Panchishkin subrepresentation of $V$ is a subspace $V^+_v \subset V$ such that
	\begin{itemize} \setlength\itemsep{0.1em}
	\item $V^+_v$ is stable under $G_{K_v}$,
	\item $V^+_v$ has all Hodge--Tate weights $\geq 1$,
	\item $V/V^+_v$ has all Hodge--Tate weights $\leq 0$.	
	\end{itemize}
	\end{defn}
	Such $V_v^+$ is unique if it exists, and gives us a local condition on $V$ which relates to the Bloch--Kato Selmer group we will see in later sections.
	\par Define the rank of $V$ as $ r(V)= \max(0,r_0(V))$ combined with the following; fix considering infinite places $v$ of $K$, let $\sigma_v$ denote complex conjugation in $\overline{K_v}$ for a real place $v$ and let $L$ be the field of definition of $V$. Then
	\begin{align*}  d_{-}(V)= & \sum_{v \mid \infty \enspace \mathrm{real}} \dim_L(V^{\sigma_v=-1}) + \sum_{v \mid \infty \enspace \mathrm{complex}} \dim_L(V)\\ 
	r_0(V)  = & \enspace d_{-}(V) - \sum_{v \mid p} \dim_L \mathrm{Fil}^0 \mathbb{D}_{\mathrm{dR}}(K_v, V) \hfill \\
		 = & \enspace d_{-}(V) - \sum_{v \mid p} \#\{ \leq 0 \textrm{ \enspace Hodge--Tate \enspace weights \enspace at \enspace} v \}.
	\end{align*}
	We will call $V$ $r$-critical if $r(V)=r$ and $r(V^*(1))=0$. If moreover there exist Panchishkin subrepresentations $V^+_v$ for all $v \mid p$, say $V$ satisfies the rank $r$ Panchishkin condition. These subrepresentations are important because of the following prediction.
	
	\begin{conj}[see \citep{LZ20}] If $V$ is $r$-critical and satisfies the rank $r$ Panchishkin condition for $r \geq 0$, there exists a collection of cohomology classes $$ c_F \in \bigwedge^r H^1_f(F,V) $$ satisfying the Euler system compatibility relation, where $F$ varies over finite abelian extensions of $K$. Moreover, $c_K$  is non-zero if and only if the associated complex $L$-function of $V$ satisfies $L^{(r)}(V^*(1),0) \neq 0$. \end{conj}
	\par This tells us that we expect to find a (rank 1) Euler system attached to any such $V$ which is 1-critical and a rank 0 Euler system (which is a $p$-adic $L$-function) when $V$ is 0-critical. For $r \geq 2$, little is known about the existence of rank $r$ Euler systems so we will only consider $r=0,1$. This $p$-adic $L$-function will conjecturally interpolate values of the complex $L$-function of $V$, although this is beyond the scope of this paper and will follow in future work.
		
	\par Back to our specific case; first let's assume $p$ splits and fix a choice of prime $\wp \mid p$. Let $\eta$ be an algebraic Hecke character of conductor dividing $\mathfrak{m}p^\infty$ for some ideal $\mathfrak{m} \subset \mathcal{O}_E$ (which we can often take to be the unit ideal) and infinity type $(s,r)$ and let $V = V_{\mathfrak{P}}(\Pi)^*$. Then the Hodge numbers of $V(\eta^{-1})$ are:
	\begin{align*} \{s, 1+b+s, 2+a+b+s \} & \enspace \mathrm{at} \enspace \wp \\
	 \{r, 1+a+r, 2+a+b+r \} & \enspace \mathrm{at} \enspace \bar{\wp}. 
	 \end{align*}

	\par We can see that $V(\eta^{-1})$ can satisfy different rank $t$ Panchishkin conditions (using $t$ for now to avoid a notational clash) depending on the choice of $\infty$-type of $\eta$, and we summarise this below --- plotting the regions where different rank $r$ Panchishkin conditions are met in Figure 1, and tabulating the ranks and subrepresentations in Figure 2. Any unlabelled regions are not $t$-critical for any $t$. Note that we are still assuming ordinarity at all primes above $p$. \vspace{5pt}\\
	
	\begin{figure}[H]
	\includegraphics[width=14cm]{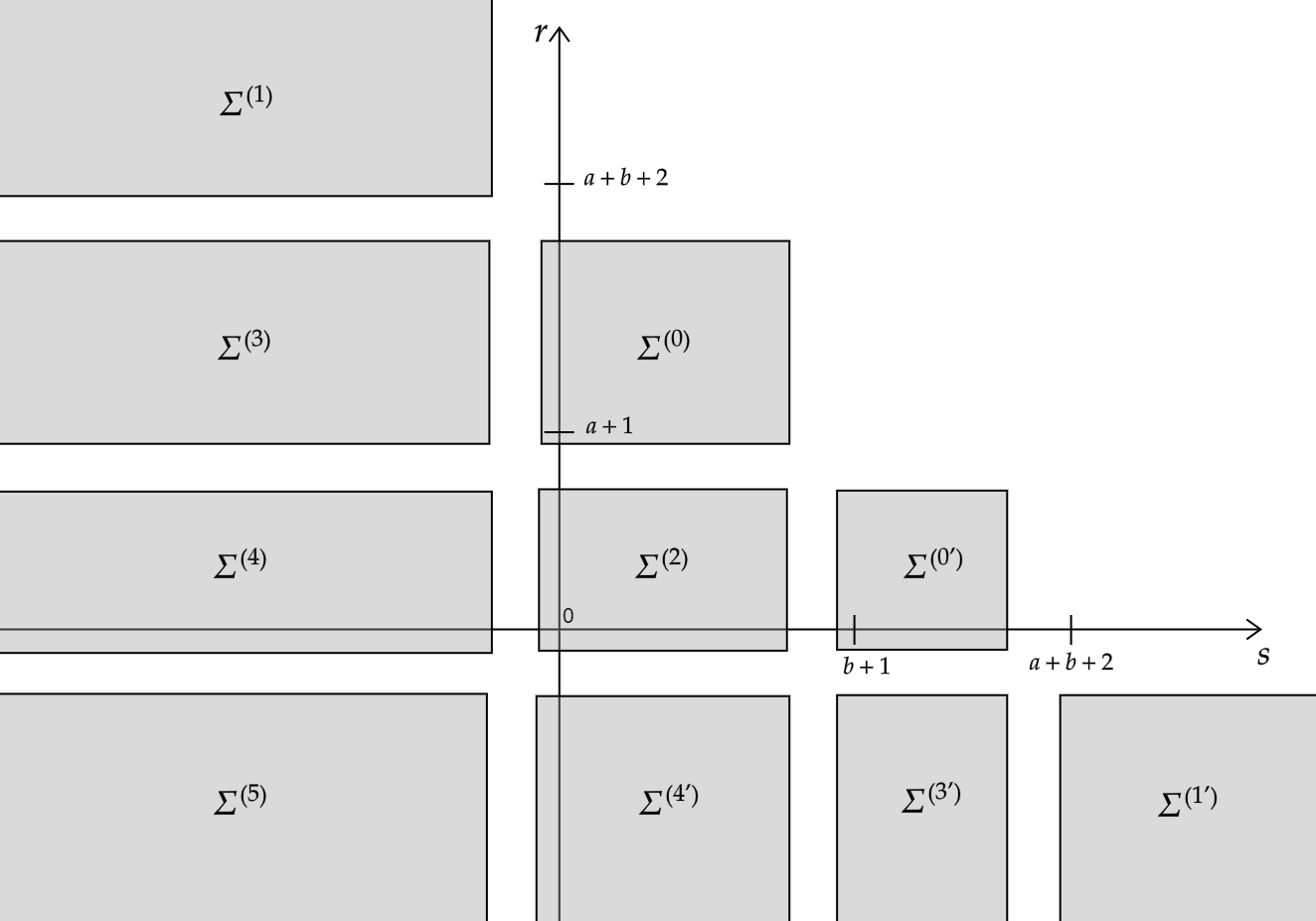}
	\caption{All possible (twisted) $r$-critical regions $V(\eta^{-1})$ as we vary $\eta_\infty$.}
	\label{fig1}
	\end{figure}
	
	\begin{figure}[H]
	{\renewcommand{\arraystretch}{1.2}
	\begin{tabular}{c|c|c|c}
		 \bf Region & \bf Rank $r$ & \bf Panchishkin subrep. at $\wp$ & \bf Panchishkin subrep. at $\bar{\wp}$  \\ \hline
		 $\Sigma^{(0)}$ & 0 & $\mathcal{F}^1$ & $\mathcal{F}^2$ \\
		 $\Sigma^{(0')}$ & 0 & $\mathcal{F}^2$ & $\mathcal{F}^1$ \\
		 $\Sigma^{(1)}$ & 0 & $V$ & 0 \\
		 $\Sigma^{(1')}$ & 0 & 0 & $V$ \\
		 $\Sigma^{(2)}$ & 1 & $\mathcal{F}^1$ & $\mathcal{F}^1$ \\
		 $\Sigma^{(3)}$ & 1 & $V$ & $\mathcal{F}^2$ \\
		 $\Sigma^{(3')}$ & 1 & $\mathcal{F}^2$ & $V$ \\
		 $\Sigma^{(4)}$ & 2 & $V$ & $\mathcal{F}^1$ \\
		 $\Sigma^{(4')}$ & 2 & $\mathcal{F}^1$ & $V$ \\
		 $\Sigma^{(5)}$ & 3 & $V$ & $V$ \\
	\end{tabular} }
		\label{fig2}
		\caption{The Panchishkin subrepresentations for $V(\eta^{-1})$ in each region.}
	\end{figure}

	\par A natural question is to ask about the case where $p$ is inert in $E$. We have the same Hodge--Tate weights at embeddings of $E_p$ into $\overline{F_{\Pi,\mathfrak{P}}}$  (rather than at the primes dividing $p$). We will have the same conditions on $r$ and $s$ for each region but with one key difference; we no longer have Panchishkin subrepresentations for every region. The rank 1 Panchishkin subrepresentation for $\Sigma^{(2)}$ still exists and is defined in the same way (which is crucial as we will see in the next subsection) but the rank 0 ones no longer exist. This will provide problems down the line, as we will later see, and a potential workaround will be discussed in \S4.1.
	
	\subsection{Rank $r$ Euler systems}
	\par From the data and the preceding conjecture, we can predict a rank 1 Euler system to exist for $V(\eta^{-1})$ when the infinity type lies in the region $\Sigma^{(2)}$ --- this has been constructed in \citep[\S12]{LSZ21} (see Theorem 2.3 below). In the split prime case we expect two conjugate rank 1 Euler systems to exist in boxes $\Sigma^{(3)}$ and $\Sigma^{(3')}$ but these have not yet been constructed yet. We can also predict two conjugate pairs of $p$-adic $L$-functions to exist when $V(\eta^{-1})$ lies in the 0-critical regions. For every pair of adjacent boxes, we want to build a bridge connecting the Euler systems, which we call a rank lowering operator. As we only have knowledge about one Euler system, we can only work with twists lying in $\Sigma^{(0)}$ and $\Sigma^{(0')}$. When $p$ is inert, we don't have as much information about local conditions attached to rank 0 boxes or associated $p$-adic $L$-functions. We will discuss this case further in \S4.1.
	\par For now we will now consider both cases where $p$ is split or inert. Fix $c \in \mathbb{Z}_{\geq1}$ coprime to $6pS$ for a finite set $S$ of primes (containing the primes at which $E$ and $\Pi$ are ramified), and let $\mathrm{Spl}_{E/\mathbb{Q}}$  be the set of rational primes splitting in $E$. Define $$\mathcal{R} = \{ \mathfrak{m} \subset \mathcal{O}_E \enspace \mathrm{squarefree, \enspace coprime \enspace to \enspace} 6pcS : \enspace l \in \mathrm{Spl}_{E/\mathbb{Q}} \implies l \nmid \mathfrak{m}  \}. $$
	\begin{thm}[\citep{LSZ21}, Theorem 12.3.1]   Assume $\Pi$ is ordinary at $p$, then there exists a lattice $T=T_\mathfrak{P}(\Pi)^* \subset V$, and a collection of classes $$ \textbf{c}^\Pi_\mathfrak{m} \in H^1_{\mathrm{Iw}}(E[\mathfrak{m}p^\infty], T_\mathfrak{P}(\Pi)^*)$$ for all $\mathfrak{m} \in \mathcal{R}$ such that:
	\begin{enumerate}[i)]
		\item For $\mathfrak{m} \mid \mathfrak{n}$, we have $$\mathrm{norm}^\mathfrak{n}_\mathfrak{m}(\textbf{c}_\mathfrak{n}^\Pi) = \left( \prod_{w \mid \frac{\mathfrak{n}}{\mathfrak{m}}} P_w(\Pi,\mathrm{Frob}_w^{-1}) \right)\textbf{c}_\mathfrak{m}^\Pi.$$
		\item For an algebraic Hecke character $\eta$ of conductor dividing $\mathfrak{m}p^\infty$ and infinity type $(s,r)$, with $0 \leq r \leq a$ and $0 \leq s \leq b$, the image of $\textbf{c}^\Pi_\mathfrak{n}$ in $H^1_{\mathrm{Iw}}(E[\mathfrak{m}p^\infty], V(\eta^{-1}))$ is the étale realisation of a motivic cohomology class.
		\item For all $v \mid p$, the projection of $\mathrm{loc}_v(\textbf{c}_\mathfrak{m}^\Pi)$ to the group $H^1_{\mathrm{Iw}}(E_v \otimes_E E[\mathfrak{m}p^\infty], T/\mathcal{F}^1_v)$ is zero.
	\end{enumerate} \end{thm}
	\begin{rmk} First note that the cohomology of our Galois representations and $\mathbb{Z}_p$-lattices inside them are related by the appendix proposition \citep[Prop B.2.4]{Rub00}, which says that if $T$ is finitely generated over $\mathbb{Z}_p$ then $H^i(G,T)$ has no divisible elements and $$ H^i(G,T) \otimes \mathbb{Q}_p \cong H^i(G,V).$$
	We will use this result without further reference. \end{rmk}
	\begin{rmk} If $T$ is unramified outside a finite set of primes $S$ containing primes above $p$, we can consider this Euler system as a class in the Iwasawa cohomology group $H^1_{\mathrm{Iw}, S}(E[p^\infty], T)$ A result of Nekov\'a\v{r} (\citep{Nek06} 8.4.4.2) tells us that $$H^i_{\mathrm{Iw}, S}(E[p^\infty], T) = H^i(E^S/E, T \otimes \Lambda(-\bf{j})), $$ where $\Gamma=\mathrm{Gal}(E[p^\infty]/E)$, $\Lambda=\mathbb{Z}_p[[\Gamma]]$ is the Iwasawa algebra, $E^S$ is the maximal extension unramified outside $S$ and $\bf{j}$:$ \Gamma \rightarrow  \Lambda^*$ is the canonical character associated to $\Lambda$. We call $\mathbb{T}=T \otimes \Lambda(-\bf{j})$ the universal twist of $T$. We can lift our Panchishkin subrepresentations to the universal twist, using the notation $\mathcal{F}_v^i\mathbb{T}$ for the image of the $\mathcal{F}_v^iT$ in $\mathbb{T}$. 
	\par $\mathbb{T}$ considers all twists of $T$ by Hecke characters $\eta$ by evaluating $\bf{j}$ at $\tilde{\eta}$ where $\tilde{\eta}$ is the Galois character of $\Lambda$ arising from $\eta$ from class field theory. For an idele $x$, we define this by $\eta([x])=\tilde{\eta} (\bf{j}$($\phi_E([x]))$, where $\phi_E$ is the global Artin map defined by sending the image of a prime $v$ in the ideles to the arithmetic Frobenius at $v$. Every such $\tilde{\eta}$ gives rise to such a Hecke character $\eta$ by the universal property of $\bf{j}$ and by surjectivity of the Artin map. We will often switch between thinking of $\eta$ as a Hecke character and a $\Lambda$-character. \end{rmk}

	\section{Bounding Selmer groups}
	\par Our main objective in constructing this Euler system is to use it to bound Selmer groups. In particular we want to compute the dimension of $H^1_f(E,V)$, the Bloch--Kato Selmer group defined by the unramified local condition at primes $v \nmid p$ and the crystalline local condition when $v \mid p$. We want to bound this by the critical values of a complex $L$-function as predicted by the Bloch--Kato conjecture, but in order to do so must first develop some more general machinery to deal with Selmer groups.
	
	\subsection{The Euler system machine}
	\par The next immediate goal is to prove a version of the Iwasawa main conjecture for our $\mathrm{GU}(2,1)$-representation $V$. We do so by using results from \citep{BO17}, which generalise the standard Euler system machinery from \citep[\S5]{MR04} to work over any base number field $K$ and to allow more general Panchishkin local conditions. In order to apply these results, we need to place some conditions on $V$ in addition to ordinarity at all primes above $p$. 
	\par Let  $\Lambda = \mathbb{Z}_p[[\Gamma]]$ where $\Gamma=\mathrm{Gal}(E[p^\infty]/E)$ where $E[p^\infty]$ is the maximal extension of $E$ unramified outside $p$.  Then $\mathrm{Gal}(E[p^\infty]/E) \cong \mathbb{Z}_p^2 \times \Delta$ where $\Delta$ is a finite abelian torsion group. We assume $\Delta$ has order coprime to $p$, which follows from the sufficient conditions that $p > 3$ and $p \nmid \# \mathrm{Cl}(E)$ where $\mathrm{Cl}(E)$ is the ideal class group of $E$. We will later see that $p \geq 5$ is a very useful assumption for other reasons.
	
	\par We sketch the argument of the last statement here. By class field theory, $$ \mathrm{Gal}(E[p^\infty]/E) \cong E^\times \backslash \mathbb{A}_{E,f}^\times / \prod_{v \nmid p} \mathcal{O}_v^\times =: \mathrm{Cl}(E,p^\infty). $$
	Then we have the exact sequence $$ 1 \rightarrow \mathcal{O}_E^\times \rightarrow \prod_{v \mid p} \mathcal{O}_v^\times \rightarrow \mathrm{Cl}(E,p^\infty) \rightarrow \mathrm{Cl}(E) \rightarrow 1 .$$
	
	Since $p$ doesn't divide the class number  $h_E = |\mathrm{Cl}(E)|$, any $p$-torsion element in $\Delta$ will come from the image of $\prod_{v \mid p} \mathcal{O}_v^\times$. Then $p$ times this element will lie in $\mathcal{O}_E^\times$ by exactness. But $E_v$ doesn't contain $p$th roots of unity ($p \neq 2$ is sufficient when $p$ splits, and we also take $p \neq 3$ for inert $p$) so this is a contradiction. Note that this argument is independent of the splitting behaviour of $p$ in $E$. This gives us an explicit size $|\Delta|=\frac{(p-1)^2|\mathrm{Cl}(E)|}{|\mathcal{O}_E^\times|}$ when $p$ splits and $\frac{(p^2-1)|\mathrm{Cl}(E)|}{|\mathcal{O}_E^\times|}$ when $p$ is inert.
	
	\par For a prime $v \mid p$ such that $\Pi_v$ is unramified, $V_{\mathfrak{P}}(\Pi) \mid_{\mathrm{Gal}(\overline{E_v}/E_v)}$ is crystalline and the eigenvalues of the power of crystalline Frobenius $\phi^{[E_v:\mathbb{Q}_v]}$ on $\mathbb{D}_{\mathrm{cris}}(V_{\mathfrak{P}}(\Pi) \mid_{\mathrm{Gal}(\overline{E_v}/E_v)})$ are the reciprocal roots of $P_v(\Pi,qX)$ by \citep[Theorem 1.2]{BLGHT11}. Using this, we can get conditions for the vanishing of certain local cohomology groups when $p$ is split (and therefore $p=q$).
	\par Consider $\eta$ as a character of $\mathrm{Gal}(E[p^\infty]/E)$ by Remark 2.5, and denote its restrictions at each prime $v \mid p$ as $\eta_v:G_{\mathbb{Q}_p}^{ab} \rightarrow \mathcal{O}_L^\times$ where we assume $L$ contains $F_\mathfrak{P}$. By local class field theory, $$G_{\mathbb{Q}_p} \cong \hat{\mathbb{Q}}_p \cong \hat{\mathbb{Z}} \times \mathbb{Z}_p^\times$$ is the decomposition into the unramified part and inertia, such that the cyclotomic character is trivial on $\hat{\mathbb{Z}}$.
	\par By considering prime-to-$p$ torsion elements in the exact sequence before the lemma, the mod $p$ character $\bar{\eta}_v$ on $\Delta$ can be pulled back to a character of $\mathbb{F}_p^\times$ for each $v$. When $p$ is split, $\mathrm{Hom}(\mathbb{F}_p^\times, \bar{\mathbb{Q}}_p^\times)$ is canonically isomorphic to $\mathbb{Z}/(p-1)\mathbb{Z}$ and generated by a fixed Teichm\"uller character $\tau$, we can write $\bar{\eta_v}=\tau^{s_v}$ for $s_v \in \mathbb{Z}/(p-1)\mathbb{Z}$. When $p$ is inert we similarly can define $s_p \in \mathbb{Z}/(p^2-1)\mathbb{Z}$ depending on a choice of identity embedding and a choice of Teichm\"uller character. These the invariant we can use to computationally determine whether the cohomology of our mod $p$ representation twisted by the reduction of $\eta$ vanishes; they can be thought of as ``mod $p$ Hodge--Tate weights" of $\eta$ although this is technically inaccurate.
	
	\begin{lem} 
		\begin{enumerate}[i)] 
		\item Suppose $p$ splits in $E$ as $p = \wp \bar{\wp}$ with $\wp$ fixed in previous notation, and $\eta$ is a Hecke character as above such that $$s_\wp \mod p-1 \notin \{0, \enspace -1, \enspace 1+b, \enspace b, \enspace 2+a+b, \enspace 1+a+b \} \subset \mathbb{Z}/(p-1)\mathbb{Z}$$, $$s_{\bar{\wp}} \mod p-1 \notin \{0, \enspace -1, \enspace 1+a, \enspace a, \enspace 2+a+b, \enspace 1+a+b \} \subset \mathbb{Z}/(p-1)\mathbb{Z}$$.
		Then for any subquotient $X$ of $\overline{T(\eta^{-1})}$, $H^i(E_v, X)=0$ for $i=0,2$ and $v=\wp, \bar{\wp}$.
		\item Suppose $p$ is inert in $E$, and $\eta$ is a Hecke character as above such that $$s_p \mod p^2-1 \notin \{0, -1, 2+a+b, 1+a+b, 4+2a+2b, 3+2a+2b \} \subset \mathbb{Z}/(p^2-1)\mathbb{Z} $$.
		Then for any subquotient $X$ of $\overline{T(\eta^{-1})}$, $H^i(E_p, X)=0$ for $i=0,2$.
		\end{enumerate}
	\end{lem}
	
	\begin{proof}
		We will first prove ($i$). By ordinarity, $H^0(E_v,\overline{T}) \neq 0 \implies H^0(E_v,\overline{\chi_i}) \neq 0$ for some $i$, where the $\chi_i$ are the crystalline characters of $E_v = \mathbb{Q}_p$ in Remark 1.8. By\citep[Prop B.4]{Con11}, all crystalline characters of $\mathrm{Gal}(\overline{\mathbb{Q}_p}/\mathbb{Q}_p)^{ab}$ are unramified twists of powers of the cyclotomic character. Say $\chi_i=\psi_i(r_i)$ where $r_i$ is the Hodge number of $\chi_i$ and $\psi_i$ is unramified, thus $\phi$ has eigenvalue $u_i=\psi_i(\mathrm{Frob}_v^{-1}) \in \mathbb{Z}_p^\times$ (with arithmetic Frobenius) on $\mathbb{D}_{\mathrm{cris}}(\psi_i)$.
		\par We can look at the following filtration of $V(\eta^{-1})$ which follows from Remark 1.8;
		\begin{equation*} V(\eta^{-1})|_{\mathrm{Gal}(\overline{E_v}/E_v)} \sim \begin{pmatrix}
				\chi_1 \cdot \eta_v^{-1} & * & * \\
				0 & \chi_2 \cdot \eta_v^{-1} & * \\
				0 & 0 & \chi_3 \cdot \eta_v^{-1} \\
			\end{pmatrix}.
		\end{equation*}
		\par If $H^0(E_v, \overline{T(\eta^{-1})})$ is non-zero for some $v|p$ then $\eta_v \equiv \chi_i$ mod $p$ for some $i$, which implies that $s_v \equiv -r_i \mod p-1$ where $r_i$ is the Hodge number of $\chi_i$ (noting that the cyclotomic character pulls back to the Teichm\"uller lift, and raising it to the power of $p-1$ is gives the trivial character mod $p$).
		\par For the $H^2$ group we follow the same argument and use local Tate duality, noting that $H^2(E_v,M)=0 \Leftrightarrow H^0(E_v,M^*(1))=0$ for all $\mathrm{Gal}(\overline{E_v}/E_v)$-modules $M$ that we work with. Moreover we can run this argument from the start for any subquotient $X$ of $\overline{T(\eta^{-1})}$  as its Hodge numbers of $X$ are a subset of the Hodge numbers of the whole representation, so we obtain the vanishing statement for $X$.
		\par For ($ii$) we run a similar argument on the same exact sequence for the invariant $s_p$. By \citep[Prop B.4]{Con11}, each crystalline characters $\chi_i$ of $\mathrm{Gal}(\overline{E_p}/E_p)^{ab}$ is of the form $$\chi_i = \chi_{\mathrm{LT}}^{a_i} \cdot\overline{\chi_{\mathrm{LT}}}^{b_i} \cdot \psi_i $$ where $\chi_{\mathrm{LT}}$ is the Lubin--Tate character attached to $E_p$ which has Hodge number 1 at the identity and 0 at the conjugate embedding, $\psi_i$ are unramified characters and $a_i$, $b_i \in \mathbb{Z}$. In fact the $a_i$ are the Hodge numbers at the identity embedding and the $b_i$ are the Hodge numbers at the conjugate embedding (and they are ordered as in Remark 1.8). We also use $\bar{\cdot}$ to denote the Galois conjugate map. 
		\par We can construct $\chi_{\mathrm{LT}}$ such that crystalline Frobenius acts as multiplication by $p^{-1}$ on the marked element $t_{\mathrm{LT}} \in \mathbb{D}_{\mathrm{cris}}(\chi_{\mathrm{LT}})$. We are avoiding the word eigenvalue as it is only a semilinear map in this instance, but it is true that $\phi^2$ is a linear map with eigenvalue $p^{-2}$ on $\mathbb{D}_{\mathrm{cris}}(\chi_{\mathrm{LT}})$. Thus $\chi_i$ acts on $t_{\mathrm{LT}}$ by multiplication by $p^{-(a_i+b_i)} u_i$ for some $u_i \in \mathcal{O}_{E_p}^\times$. We now obtain the result of $(ii)$ by computing the $a_i+b_i$ and repeating the argument of $(i)$.
	\end{proof}

	\par Henceforth we will assume that $\eta$ is such that the hypothesis of Lemma 3.1 holds, and so the corresponding degree 0 and 2 cohomology groups vanish.
	
	\begin{rmk} Note that the number of mod $p$ characters on $\Lambda$ is indexed by $\Delta$ which has size $\delta = \frac{(p-1)^2|\mathrm{Cl}(E)|}{|\mathcal{O}_E^\times|}=O(p^2)$, and the number of bad characters mod $p$ is bounded independently of $p$ by the lemma. Thus as $p$ gets large, we can guarantee the existence of $\eta$ satisfying the hypothesis. \end{rmk}
		
	
	\par Having assumed $p \nmid \delta$ we have a decomposition of the Iwasawa algebra $$ \Lambda \cong \bigoplus_{\bar{\eta}} \Lambda_{\bar{\eta}} $$ where $\bar{\eta}$ ranges over the characters of $\Delta$. This also gives us a decomposition of $\Lambda$-modules into a direct sum of $\Lambda(\bar{\eta})$-modules. For $\mathbb{T}$ we will denote this by $\mathbb{T}=\bigoplus_{\bar{\eta}} \mathbb{T}_{\bar{\eta}}$. We have maximal ideals $\mathfrak{m}_{\bar{\eta}}$ of $\Lambda$, and for each $\bar{\eta}$ we set $\overline{\mathbb{T}}_{\bar{\eta}}=\mathbb{T}/\mathfrak{m}_{\bar{\eta}}$. The above vanishing results tell us that $H^j(E_v, \overline{\mathbb{T}}_{\bar{\eta}})=H^j(E_v, \mathbb{T}) = 0$ for each $\bar{\eta}$, each $v \mid p$ and $j=0,2$.

	\par We further need to assume some hypotheses known collectively as ``big image" criteria. For our choice of $\eta$:
	\begin{itemize}
		\item Either $p \geq 5$ or $\mathrm{Hom}_{\mathbb{F}_p[[G_E]]}(\overline{\mathbb{T}}_{\bar{\eta}},\overline{\mathbb{T}}_{\bar{\eta}}^*(1)) = 0$, 
		\item $\overline{\mathbb{T}}_{\bar{\eta}}$ is absolutely irreducible as a $G_E$-module,
		\item There is an element $\tau \in G_E$ such that $\mathbb{T}/(\tau-1)\mathbb{T}$ is free of rank 1 and $\tau$ acts trivially on $p$-power roots of unity.
	\end{itemize}
	These are expected to hold for all but finitely many $p$ for any non-endoscopic $\Pi$ (similar results have already been proven for $\mathrm{GL_2}$ and $\mathrm{GSp}_4$, e.g.\ see \citep{DZ20}) so these are fairly safe assumptions.

	\par Now we will develop some notation for Selmer groups which we will need to state the first main result. We define a ``rank 1" local condition as follows: let $\mathcal{F}_+$ be the unramified local condition at primes $v \nmid p$ and the image of the rank 1 Panchishkin subrepresentation at primes above $p$ (as in Figure 2 region $\Sigma^{(2)}$). Precisely,
	$$
		H^1_{\mathcal{F}_+}(E_v, \mathbb{T}) = 
		\begin{cases}
			H^1_{\mathrm{ur}}(E_v, \mathbb{T}) \enspace \mathrm{if \enspace v \nmid p} \\
			\mathrm{im}(H^1(E_v, \mathcal{F}_v^1\mathbb{T}) \rightarrow H^1(E_v,\mathbb{T})) \enspace \mathrm{if \enspace v \mid p}.
		\end{cases} $$
	Then define $H^1_{\mathcal{F}_+}(E,\mathbb{T})$ as the set of cohomology classes lying locally in $H^1_{\mathcal{F}_+}(E_v,\mathbb{T})$ for each $v$. We also define the $\mathcal{F}^*_+$ as the dual set of local conditions on the representation $\mathbb{T}^*(1)$ (with respect to local Tate duality). We can define such Selmer groups for $\mathbb{T}_{\bar{\eta}}$ similarly.  
	\par We can now formulate one divisibility in the rank 1 Iwasawa main conjecture, which we will prove below under the assumptions we have listed throughout the paper, in particular Lemma 3.1. In the case of inert $p$ we still assume that $\eta$ is such that the relevant $H^0$ and $H^2$ groups vanish, but we no longer have a tidy condition on $\eta$.
	
	\par Given a finitely generated torsion module $M$ over a Noetherian Krull domain $R$, we can define its characteristic ideal as follows; the structure theorem of Iwasawa theory (e.g.\ see \citep[\S2]{BBL14}) tells us that there are pseudo-null $R$-modules $P$, $Q$ (i.e.\ the localisations of $P$ and $Q$ at all height 1 primes of $R$ are zero) such that the following sequence is exact for some $n \in \mathbb{N}$ and some height 1 primes $\mathfrak{p}_i$ of $R$ and integers $e_i$; $$ 0 \rightarrow P \rightarrow M \rightarrow \prod_{i=1}^n R/\mathfrak{p}_i^{e_i} \rightarrow Q \rightarrow 0 .$$
	Then we define the characteristic ideal of $M$ (over $R$) by $$ \mathrm{char}_R(M) = \prod_{i=1}^n \mathfrak{p}_i^{e_i} $$ and we will say $\mathrm{char}_R(M)=0$ when $M$ is non-torsion. Using our Euler system and an adaptation of classical Euler system machinery, we develop a divisibility of characteristic ideals over the Iwasawa algebra $\Lambda$, which is part of the Iwasawa main conjecture for GU(2,1). We will denote the projection maps onto each component by $e_{\bar{\eta}}$. Since $p \nmid |\Delta|$, this commutes with taking cohomology.

	\begin{thm} Assume that $V$ is ordinary at $p$, satisfies the big image criteria and $\eta$ is a Hecke character satisfying the vanishing conditions of Lemma 3.1. Suppose further that the bottom class of the Euler system satisfies $e_{\bar{\eta}} c^\Pi_1 \neq 0$. Then the following statements hold:
	\begin{enumerate}[i)]
		\item $e_{\bar{\eta}} H^1_{\mathcal{F}^*_+}(E,\mathbb{T}^*(1))^\vee$ is $\Lambda_{\bar{\eta}}$-torsion,
		\item $e_{\bar{\eta}} H^1_{\mathcal{F}_+}(E,\mathbb{T})$ is torsion-free of $\Lambda_{\bar{\eta}}$-rank 1,
		\item $\mathrm{char}_{\Lambda_{\bar{\eta}}}\left( e_{\bar{\eta}} H^1_{\mathcal{F}^*_+}(E,\mathbb{T}^*(1))^\vee  \right) \mid \mathrm{char}_{\Lambda_{\bar{\eta}}}\left( \dfrac{e_{\bar{\eta}} H^1_{\mathcal{F}_+}(E,\mathbb{T})}{ e_{\bar{\eta}} c^\Pi} \right). $
	\end{enumerate} \end{thm}
	\begin{proof}
		Apply \citep[Theorem 3.6]{BO17} with $R=\Lambda_{\bar{\eta}}$, which is isomorphic to $\mathbb{Z}_p[[X_1,X_2]]$.
	\end{proof}

	\subsection{Selmer complexes}
	\par Let S be a finite set of primes of $E$ containing all primes above $p$, and $E^S$ defined as before. Denote $G_{E,S} = \mathrm{Gal}(E^S/E)$. Let $M$ be a finitely generated module of a ring $R$ with a continuous action of $G_{E,S}$. \begin{defn} A Selmer structure $\Delta = (\Delta_v)_{v \in S}$ for $M$ is a collection of homomorphisms $$\Delta_v: U_v^+ \rightarrow C^\bullet(E_v,M)$$ from a complex $U_v^+$ into the complex of continuous $M$-valued cochains. \end{defn} Following work of \citep{Nek06}, summarised in \citep[\S11.2]{KLZ17} we can define the associated Selmer complex $\widetilde{\mathrm{R\Gamma}}(\mathcal{O}_{E,S},M;\Delta)$. We can compute the cohomology of this Selmer complex, which we will denote by $\widetilde{H}^i(\mathcal{O}_{E,S},M;\Delta)$. We will mainly use this for $i=1,2$ when $M$ is a twist of $T$, the $\mathbb{Z}_p$-lattice inside $V=V_\mathfrak{P}(\Pi)^*$.
	\par A Selmer structure is called simple if for all $v$ the map $$i_v: H^i(U_v^+) \rightarrow H^i(E_v,M) $$ is an isomorphism for $i=0$, injection for $i=1$ and zero for $i=2$. In this case, the local condition is determined by the subspace $i(H^1(U_v^+)) \subseteq H^1(E_v,M)$.
	\par We define three Selmer structures which we will use; first define the Bloch--Kato structure $\Delta^{BK}$ by the saturation of the unramified condition at each $v \nmid p$ and by the crystalline local condition for $v \mid p$. We also define $\Delta^1$ by the unramified local condition for $v \nmid p$ and above $p$ by the inclusions $$C^\bullet(E_v, T^+) \hookrightarrow C^\bullet(E_v,T),$$ where $T^+$ is a $\mathbb{Z}_p$-lattice in the Panchishkin subrepresentation of $V$ in the region $\Sigma^{(2)}$. We similarly define $\Delta^0$ for the region $\Sigma^{(0)}$. Note that these are both simple local conditions by the vanishing of local $H^0$ groups assumed in \S3.1; as the result shows every $G_{E_v}$-subquotient of $\overline{\mathbb{T}}$ has no non-trivial $G_{E_v}$ invariant vectors and the required results for global cohomology groups follow. Using \citep[Prop 11.2.9]{KLZ17}, we see that for each of these simple structures,
	$$ \widetilde{H}^0(\mathcal{O}_{E,S},M;\Delta^?) = H^0(\mathcal{O}_{E,S},M)$$ 
	$$\widetilde{H}^1(\mathcal{O}_{E,S},M;\Delta^?) = H^1_{\Delta^?}(\mathcal{O}_{E,S},M)$$
	$$\widetilde{H}^2(\mathcal{O}_{E,S},M;\Delta^?) = H^1_{\Delta^?*}(\mathcal{O}_{E,S},M^*(1))^\vee $$
	where $H^i(\mathcal{O}_{E,S},-)$ is cohomology of $G_{E,S}$ modules, $? \in \{0,1,\mathrm{BK} \}$. Note that $H^1_{\mathrm{BK}}$ denotes the Bloch--Kato Selmer group, also denoted $H^1_f$ in many references and $H^1_\Delta$ denotes the usual Selmer group defined by the simple condition $\Delta$.
	Thus we can reformulate the last part of Theorem 3.3 as; 
	$$ \mathrm{char}_{\Lambda_{\bar{\eta}}} \left( e_{\bar{\eta}} \widetilde{H}^2(\mathcal{O}_{E,S},\mathbb{T};\Delta^1) \right) \bigm\vert \mathrm{char}_{\Lambda_{\bar{\eta}}} \left( \dfrac{e_{\bar{\eta}} \widetilde{H}^1(\mathcal{O}_{E,S},\mathbb{T};\Delta^1)}{ e_{\bar{\eta}}c^\Pi_1} \right). $$
	
	\subsection{Descending to integral level}
	\par We will develop a useful base change property of Selmer complexes, then by following the structure of arguments in \citep[\S11]{KLZ17} we can apply Theorem 3.3 by descent to get a bound on the size of the Selmer group. First we need a technical descent lemma; for the following we will take $\mathcal{O}=\mathcal{O}_{F_\Pi}$, so that everything we work with has the structure of an $\mathcal{O}$-module.
	\begin{lem} Let $R$ be a noetherian Krull domain, and $S=R[[T]]$. Let $C$ be a perfect complex of $S$-modules supported in degrees \{0,1,2\} such that $H^0(C)$ and $H^2(C)$ are $S$-torsion and there exists $z \in H^1(C)$ such that $H^1(C)/z$ is $S$-torsion. Then the formation of the fractional ideal $$\mathcal{I}= \dfrac{\mathrm{char}_S \left(H^1(C)/z \right)}{\mathrm{char}_S \left(H^2(C) \right)\mathrm{char}_S \left(H^0(C) \right)} $$ commutes with derived base change. That is, if $\tau:S \rightarrow R$ is a character and $\ker(\tau)$ is not in the support of the modules $H^0(C)$, $H^2(C)$ or $H^1(C)/z$, 
	$$\tau(\mathcal{I})= \dfrac{\mathrm{char}_R ( H^1(C^\prime )/\tau(z)) }{\mathrm{char}_R ( H^2(C^\prime ))\mathrm{char}_R ( H^0(C^\prime ) ) } $$
	where $C^\prime= C \otimes_{S, \tau} R$. \end{lem}
	\begin{proof} Let $\mathfrak{p} = \ker \tau$, which is a height 1 prime of $S$ by construction and thus has a principal generator $a$. We let $M_0=H^0(C)$, $M_2=H^2(C)$ and $M_1=H^1(C)/z$ and analogously define $M^\prime_0=H^0(C^\prime)$, $M^\prime_2 = H^2(C^\prime)$ and $M_1^\prime=H^1(C^\prime)/\tau(z)$.
	\par Consider the exact sequence $$0 \rightarrow C \xrightarrow{\times a} C \rightarrow C/aC \rightarrow 0.$$ 
	Notice that $C/aC = C \otimes S/\mathfrak{p} = C \otimes S/\ker \tau \cong C \otimes_\tau R = C^\prime$ so this induces a long exact sequence of cohomology of $R$-modules:
	\[\begin{tikzcd}
		0 & {H^0(C)} & {H^0(C)} & {H^0(C^\prime)} \\
		& {H^1(C)} & {H^1(C)} & {H^1(C^\prime)} \\
		{} & {H^2(C)} & {H^2(C)} & {H^2(C^\prime)} & 0
		\arrow["{\times a}"', from=2-2, to=2-3]
		\arrow[from=2-3, to=2-4]
		\arrow["{\times a}"', from=3-2, to=3-3]
		\arrow[from=3-3, to=3-4]
		\arrow[from=3-4, to=3-5]
		\arrow[from=2-4, to=3-2]
		\arrow[from=1-1, to=1-2]
		\arrow["{\times a}"', from=1-2, to=1-3]
		\arrow[from=1-3, to=1-4]
		\arrow[from=1-4, to=2-2]
	\end{tikzcd}\]
	From this we get three exact sequences:
	\begin{eqnarray}  
	0 \rightarrow M_2/aM_2 \xrightarrow{\tau} M_2^\prime \rightarrow 0 \\
	0 \rightarrow M_0/aM_0 \xrightarrow{\tau} M^\prime_0 \rightarrow H^1(C)[a] \rightarrow 0 \\
	0 \rightarrow H^1(C)/aH^1(C) \xrightarrow{\tau} H^1(C^\prime) \rightarrow M_2[a] \rightarrow 0
	\end{eqnarray}
	Note that by the restrictions on the support of $\ker \tau$, (1) and (2) are sequences of torsion $R$-modules. Since $z \in H^1(C)$ is non-torsion, $H^1(C)[a] = M_1[a]$ in (2). By taking the quotient by $z$ in the first term of (3), this becomes a sequence of torsion $R$-modules $$0 \rightarrow M_1/aM_1 \rightarrow M^\prime_1 \rightarrow M_2[a] \rightarrow 0 .$$
	Using \citep[Theorem 1.1]{BBL14} (a generalisation of \citep[14.15]{Kat04}), we have that for any $S$-module $M$ and any non-zero divisor $a \in S$,  
	$$ \mathrm{char}_R(M[a])\tau(\mathrm{char}_S(M))=\mathrm{char}_R(M/aM) $$
	Applying this to each $M_i$, and combining this with additivity of $\mathrm{char}_R$ in the three exact sequences above gives us the result. 

  	\end{proof}

	\begin{cor} Let $\eta:\Lambda \rightarrow \mathcal{O}^\times$ be a character satisfying the vanishing conditions of Lemma 3.1 such that $\ker \eta$ is not in the support of $e_{\bar{\eta}} \widetilde{H}^1(\mathcal{O}_{E,S},\mathbb{T}_{\bar{\eta}};\Delta^1)/e_{\bar{\eta}} c^\Pi$. Assume that $V$ is ordinary at $p$, satisfies the big image criteria. Suppose further that $c^\Pi_1$ has non-zero image in $H^1(\mathcal{O}_{E,S},T(\eta^{-1}))$. Then,
	\begin{enumerate}[i)]
		\item $\widetilde{H}^2(\mathcal{O}_{E,S},T(\eta^{-1});\Delta^1)$ is finite,
		\item $\widetilde{H}^1(\mathcal{O}_{E,S},T(\eta^{-1});\Delta^1)$ is free of $\mathcal{O}$-rank 1,
		\item $\#\left( \widetilde{H}^2(\mathcal{O}_{E,S},T(\eta^{-1});\Delta^1) \right) \leq \#\left( \dfrac{\widetilde{H}^1(\mathcal{O}_{E,S},T(\eta^{-1});\Delta^1)}{ \eta(c^\Pi_1)} \right)$.
	\end{enumerate}  \end{cor}
	\begin{proof}
	We can decompose $\eta$ as $$\Lambda \xrightarrow{e_{\bar{\eta}} } \Lambda_{\bar{\eta}}  \xrightarrow{\eta_1} \Lambda^\prime  \xrightarrow{\eta_2} \mathcal{O} $$ where $\Lambda^\prime $ is isomorphic to $\mathcal{O}^\prime[[T]]$ for some $\mathcal{O}^\prime$ integral extension of $\mathbb{Z}_p$ inside $\mathcal{O}$. Apply Lemma 3.5  first to the complex $C=e_{\bar{\eta}} \widetilde{\mathrm{R\Gamma}}(\mathcal{O}_{E,S},\mathbb{T};\Delta^1)$ with $\tau=\eta_1$, $S=\Lambda$ and $R=\Lambda^\prime $ and second to the complex $C=\widetilde{\mathrm{R\Gamma}}(\mathcal{O}_{E,S},\mathbb{T} \otimes_{\Lambda, \eta_1} R;\Delta^1)$ with $\tau=\eta_2$, $S=\Lambda^\prime$ and $R=\mathcal{O}$. The assumption of vanishing of $H^0$ groups and torsion modules follow by previous results. The vanishing of $H^3$ groups follows from vanishing of $H^0$ and a generalised form of Poitou-Tate duality (see \citep[Theorem 11.2.7]{KLZ17}). By Theorem 3.3, $\mathcal{I} \subseteq \Lambda$ is a proper ideal and so $\eta(\mathcal{I}) \subseteq \mathcal{O}$ is too. Thus any prime not in the support of $\dfrac{\widetilde{H}^1(\mathcal{O}_{E,S},T(\eta^{-1});\Delta^1)}{ \eta(c^\Pi_1)}$ is not in the support of $\widetilde{H}^2(\mathcal{O}_{E,S},T(\eta^{-1});\Delta^1)$. We know they are finite, so this proves our result. \end{proof}
 
	\section{Motivic $p$-adic $L$-functions}
	
	\par Our next immediate goal is to start tying our results to some sort of $p$-adic $L$-function. The standard tool for doing this is to use a version of Perrin-Riou's regulator map to map the first class $c^\Pi_1$ of our Euler system to a $p$-adic measure, which we call a ``motivic $p$-adic $L$-function". The right map has been constructed in \citep[\S4]{LZ14}, and we by choosing a basis carefully we use it to define a Coleman map when $p$ is split in $E$:
	$$ \mathrm{Col}^\Pi: H^1(E_\wp, \mathbb{T}^1/\mathbb{T}^0) \longrightarrow \Lambda \otimes \mathbb{D}_{\mathrm{cris}}(T^1/T^0)  $$ where $\mathbb{T}^i$ and $T^i$ are the images of the rank $i$ Panchishkin local condition in $\mathbb{T}$ and $T$ for $i=0,1$ respectively. Note that our choice of $\wp \mid p$ at which we restrict is fixed as before; the unique prime of $E$ where local rank 0 and rank 1 conditions differ. The bottom Euler system class lies in the domain, so we can apply the Coleman map to it. We want to fix a $\mathbb{Z}_p$ basis $\{ \xi \}$ of the crystalline character $\mathbb{D}_{\mathrm{cris}}(T^1/T^0)$ in order to renormalise the Coleman map; this integral version is denoted $\mathrm{Col}^{\Pi,\xi}$ and has image contained in $\Lambda$. The choice of $\xi$ is not important now, but in later work the right $\xi$ may need to be chosen to compare our map to existing $p$-adic analytic constructions.

	\par Define $L_p^*(\Pi, 1 + \bf{j})$ := $ \mathrm{Col}^\Pi(c^\Pi)$ to be the ``motivic $p$-adic $L$-function''. Taking the integral version, $$\mathrm{Col}^{\Pi,\xi}(c^\Pi) = \frac{L_p^*(\Pi, 1 + \bf{j})}{\Omega} $$ for some non-zero scalar $\Omega$ which will depend on our choice of $\xi$. With some exact sequences, we can restate the main conjecture to include the ideal generated by this $L$-function.
	
	\begin{thm}
	Assume that $p$ splits in $E$ and the hypotheses of Theorem 3.3, i.e.\ $V$ is ordinary at $p$, satisfies the big image criteria and $\eta$ is a Hecke character satisfying the hypotheses of Lemma 3.1. Assume further that $e_{\bar{\eta}} L^*_p(\Pi, 1+ \textbf{j}) \neq 0$. Then,
		\begin{enumerate}[i)]
			\item $e_{\bar{\eta}} \widetilde{H}^1(\mathcal{O}_{E,S},\mathbb{T};\Delta^0) = 0 $,
			\item  $e_{\bar{\eta}} \widetilde{H}^2(\mathcal{O}_{E,S},\mathbb{T};\Delta^0)$ is torsion,
			\item $ \mathrm{char}_{\Lambda_{\bar{\eta}} } \left(e_{\bar{\eta}}  \widetilde{H}^2(\mathcal{O}_{E,S},\mathbb{T};\Delta^0)\right) \mid \frac{_{\bar{\eta}} L^*_p(\Pi, 1+ \bf{j})}{\Omega}$.
		\end{enumerate}
	\end{thm}
	\begin{proof}
	We have an exact triangle from the formation of Selmer complexes, written down in \citep[11.2.1]{KLZ17}:
	$$\widetilde{R\Gamma}(\mathcal{O}_{E,S}, \mathbb{T}; \Delta^0) \rightarrow \widetilde{R\Gamma}(\mathcal{O}_{E,S}, \mathbb{T}; \Delta^1) \rightarrow R\Gamma(E_\wp, \mathbb{T}^1/\mathbb{T}^0) \rightarrow \dots $$
	\par First note that $e_{\bar{\eta}}\widetilde{H}^1(\mathcal{O}_{E,S}, \mathbb{T}; \Delta^0)$ injects into $e_{\bar{\eta}}\widetilde{H}^1(\mathcal{O}_{E,S}, \mathbb{T}; \Delta^1)$ which is torsion free of rank 1 by Theorem 3.3. However, the latter has an element $e_{\bar{\eta}}c^\Pi$ whose image in $e_{\bar{\eta}}H^1(E_\wp, \mathbb{T}^1/\mathbb{T}^0)$ is non-torsion. This gives us part $(i)$.
	\par From the exact triangle we get the following exact sequence,
	$$ 0 \rightarrow \ \frac{e_{\bar{\eta}}\widetilde{H}^1(\mathcal{O}_{E,S}, \mathbb{T}; \Delta^1)}{ e_{\bar{\eta}} c^\Pi} \rightarrow \frac{e_{\bar{\eta}}H^1(E_\wp, \mathbb{T}^1/\mathbb{T}^0)}{ e_{\bar{\eta}}c^\Pi} \rightarrow e_{\bar{\eta}} \widetilde{H}^2(\mathcal{O}_{E,S}, \mathbb{T}; \Delta^0) \rightarrow e_{\bar{\eta}}\widetilde{H}^2(\mathcal{O}_{E,S}, \mathbb{T}; \Delta^1) $$
	using the above as well as vanishing of local $H^0$. The last map is actually surjective due to the vanishing assumptions of local $H^2$ groups. This gives us part $(ii)$ as the first two modules are torsion. Applying the divisibility of characteristic ideals from the rank 1 main conjecture, we get
	$$ \mathrm{char}_{\Lambda_{\bar{\eta}}} \left(e_{\bar{\eta}} \widetilde{H}^2(\mathcal{O}_{E,S}, \mathbb{T}; \Delta^0)\right) \mid \mathrm{char}_{\Lambda_{\bar{\eta}}} \left(\frac{e_{\bar{\eta}} H^1(E_\wp, \mathbb{T}^1/\mathbb{T}^0)}{ e_{\bar{\eta}} c^\Pi} \right). $$
	\par We also have an exact sequence induced by the Coleman map:
	$$ 0 \rightarrow \frac{e_{\bar{\eta}} H^1(E_\wp, \mathbb{T}^1/\mathbb{T}^0)}{ e_{\bar{\eta}} c^\Pi} \rightarrow \frac{\Lambda_{\bar{\eta}}}{e_{\bar{\eta}} \mathrm{Col}^{\Pi,\xi}(c^\Pi)} \rightarrow \mathrm{coker}(e_{\bar{\eta}} \circ \mathrm{Col}^{\Pi,\xi}) \rightarrow 0 .$$
	By Remarks 8.2.4 and 8.2.5 of \cite{KLZ17}, if we assume 1 is not an eigenvalue of crystalline Frobenius on $V(\eta^{-1})$ (which follows from the statement of Lemma 3.1) then $\mathrm{coker}(e_{\bar{\eta}} \circ \mathrm{Col}^{\Pi,\xi})$ is a pseudonull $\Lambda$-module, and so
	\begin{align*} 
		\mathrm{char}_{\Lambda_{\bar{\eta}}} \left(e_{\bar{\eta}} \widetilde{H}^2(\mathcal{O}_{E,S}, \mathbb{T}; \Delta^0)\right) &  \mid \enspace \mathrm{char}_{\Lambda_{\bar{\eta}}} \left( \frac{e_{\bar{\eta}}H^1(E_\wp, \mathbb{T}^1/\mathbb{T}^0)}{e_{\bar{\eta}} c^\Pi} \right) \hfill \\
	  	 & = \mathrm{char}_{\Lambda_{\bar{\eta}}} \left(\frac{e_{\bar{\eta}}H^1(E_\wp, \mathbb{T}^1/\mathbb{T}^0)}{e_{\bar{\eta}} c^\Pi} \right)\mathrm{char}_{\Lambda_{\bar{\eta}}} (\mathrm{coker}(e_{\bar{\eta}} \circ \mathrm{Col}^{\Pi,\xi})) \\ & = \enspace \left(\frac{e_{\bar{\eta}} L^*_p(\Pi, 1+ \bf{j})}{\Omega}\right). 
	\end{align*}
	Thus we have the required divisibility.
	\end{proof}
	
	\begin{rmk}
		When $p$ is inert, we don't have a rank 0 Panchishkin subrepresentation or a Coleman map as above. So in order to construct an analogue of the motivic $p$-adic $L$-function we will have to resort to other methods, which we will discuss in \S4.1.
	\end{rmk}

 	\par This has established a strong link between the Selmer groups we are considering and the $L$-function as a distribution, but we would still like to go further and gather information about special values of the $L$-function after specialising $\mathbb{T}$ by a character $\eta$ with all the conditions we have placed above. We will define $L^*_p(\Pi, 1+ \eta) := \eta(L^*_p(\Pi, 1+\textbf{j})) \in \mathcal{O}$ as the $L$-value at $\eta$. Similar to \S3.2, we want to prove a descended version of Theorem 4.1 at a finite level. 
 	
 	\begin{thm}
 	 Let $\eta:\Lambda \rightarrow \mathcal{O}^\times$ be a character satisfying the hypotheses of Lemma 3.1 such that $\ker \eta$ is not in the support of $L^*_p(\Pi, 1 + \textbf{j})$. Assume that $V$ is ordinary at $p$, satisfies the big image criteria. Then,
 		\begin{enumerate}[i)]
 			\item We have \begin{align*} \mathrm{rank}_\mathcal{O} \widetilde{H}^1(\mathcal{O}_{E,S},T(\eta^{-1});\Delta^0) & = \mathrm{rank}_\mathcal{O} \widetilde{H}^2(\mathcal{O}_{E,S},T(\eta^{-1});\Delta^0)\\
 				& \leq \mathrm{ord}_{\textbf{j} =\eta} L_p^*(\Pi, 1+\textbf{j}).
			\end{align*}
		\item If $L_p^*(\Pi, 1+\eta) \neq 0$, then we have $\widetilde{H}^1(\mathcal{O}_{E,S},T(\eta^{-1});\Delta^0)=0$, and $\widetilde{H}^2(\mathcal{O}_{E,S},T(\eta^{-1});\Delta^0)$ is a finite $\mathcal{O}$-module whose length is bounded above by $$v_\mathfrak{P}\left(\frac{L_p^*(\Pi, 1+\eta)}{\Omega}\right). $$
 				
		\end{enumerate}
 	\end{thm}
 	\begin{proof}
 		First we will compute the Euler characteristic $\chi$ of the complex $C=\widetilde{H}^\bullet(\mathcal{O}_{E,S},T(\eta^{-1});\Delta^0)$. By the exact triangle of \citep[11.2.2a]{KLZ17}, we compute this using Tate's local and global Euler characteristic formulae
 		\begin{eqnarray*} \chi(C) & = \sum_v\chi(H^\bullet(E_v, T/T^0)-\chi(H^\bullet(\mathcal{O}_{E,S},T)\\
 			& = - \mathrm{rank}_\mathcal{O}(T^0_\wp) - \mathrm{rank}_\mathcal{O}(T^0_{\bar{\wp}}) + \mathrm{rank}_\mathcal{O}(T)\\
 			& = -2 - 1 + 3 = 0
 		\end{eqnarray*}
 		where $T^0$ is the image of the rank 0 Panchishkin subrepresentation in $T$, so $T/T^0$ is defined only by local conditions at primes above $p$ (hence the sum of local Euler characteristics only has 2 non-zero terms).
 		\par By vanishing of global $H^3$ groups established in the proof of Corollary 3.6, this tells us that $$ \mathrm{rank}_\mathcal{O} \widetilde{H}^1(\mathcal{O}_{E,S},T(\eta^{-1});\Delta^0) = \mathrm{rank}_\mathcal{O} \widetilde{H}^2(\mathcal{O}_{E,S},T(\eta^{-1});\Delta^0). $$
 		Moreover by the vanishing of global $H^0$ groups, $H^1(\mathcal{O}_{E,S},T(\eta^{-1}))$ is a free $\mathcal{O}$-module. As $\Delta^0$ is a simple local condition, $\widetilde{H}^1(\mathcal{O}_{E,S},T(\eta^{-1});\Delta^0)$ is also free.
 		\par If $L_p^*(\Pi, 1+\textbf{j}) =0$ then it has infinite order of vanishing and the theorem is trivially true, so suppose not. By exact sequence (1) of Lemma 3.5 applied to our complex $C$ (and applied again to $C^\prime = C \otimes_{\Lambda, \eta} \Lambda^\prime$ using notation from Corollary 3.6), $\widetilde{H}^2(\mathcal{O}_{E,S},T(\eta^{-1});\Delta^0)$ is the maximal subquotient of $\widetilde{H}^2(\mathcal{O}_{E,S},\mathbb{T};\Delta^0)$ on which $\Gamma$ acts via $\eta$. Thus $$ \mathrm{rank}_\mathcal{O} \widetilde{H}^2(\mathcal{O}_{E,S},T(\eta^{-1});\Delta^0) = \mathrm{ord}_{\textbf{j}=\eta}\mathrm{char}_\Lambda( \widetilde{H}^2(\mathcal{O}_{E,S},\mathbb{T};\Delta^0)).  $$
 		By applying the divisibility in Theorem 4.1, we get the required inequality for $(i)$.
 		\par Now assume $L_p^*(\Pi, 1+\eta) \neq 0$, so the order of vanishing in part $(i)$ is 0. Then the $\widetilde{H}^i(\mathcal{O}_{E,S},T(\eta^{-1});\Delta^0)$ for $i=1,2$ are both finite $\mathcal{O}$-modules. For $i=1$ it is also free, and hence it is 0. To get a bound on the size of $\widetilde{H}^2$, we repeat part of the argument of Corollary 3.6 but with $M_0=1$, $M_2=e_{\bar{\eta}}\widetilde{H}^2(\mathcal{O}_{E,S},\mathbb{T};\Delta^0)$ and $M_1=\Lambda_{\bar{\eta}}/e_{\bar{\eta}}\mathrm{Col}^{\Pi, \xi}(c^\Pi)$; in particular, exact sequence (1) and the result from \citep{BBL14} combined with non-vanishing of $L^*(\Pi, 1+\textbf{j})$ under $\eta$ to give the descended divisibility $$\mathrm{char}_\mathcal{O}(\widetilde{H}^2(\mathcal{O}_{E,S},T(\eta^{-1});\Delta^0)) \mid \left( \frac{L^*_p(\Pi, 1+\eta)}{\Omega} \right). $$
 		Since $\mathcal{O}$ is a local ring with maximal ideal $\mathfrak{P}$, for any $\mathcal{O}$-module $M$, $\mathrm{char}_\mathcal{O}(M)=\mathfrak{P}^{l(M)}$ where $l(-)$ is the length as an $\mathcal{O}$-module. So the divisibility over $\mathcal{O}$ gives us part $(ii)$.
 	\end{proof}
 
 	\par Now we have bounded a Selmer group with $\Delta^0$ local conditions in terms of the $p$-adic $L$-function, we want to use this to obtain a result for the Bloch--Kato Selmer group. Thankfully, this is closely related to the Selmer groups we have already worked with, as shown by the following result.
 	\begin{cor}
 		Suppose $\eta: \Lambda \rightarrow \mathcal{O}^\times$ is a character satisfying hypothesis of Theorem 4.3, with infinity type (as a Hecke character) lying in region $\Sigma^{(0)}$ in Figure 1. Then $\widetilde{H}^2(\mathcal{O}_{E,S},T(\eta^{-1});\Delta^{BK})$ is finite.
 	\end{cor}
 	\begin{proof}
 	The condition that $\eta$ lies in the required range means that $T(\eta)$ is de Rham at places above $p$, and thus $\Delta^{BK}$ can be defined. We aim to construct maps $$\widetilde{H}^i(\mathcal{O}_{E,S},T(\eta^{-1});\Delta^0) \rightarrow \widetilde{H}^i(\mathcal{O}_{E,S},T(\eta^{-1});\Delta^{BK})$$ for $i=0,1,2$ with finite kernel and cokernel. By exact triangle of \citep[11.2.1]{KLZ17} we have maps between local complexes at primes $v$ of $E$. For $i=0$ and $i=2$ the maps are trivially isomorphisms at each $v$ by our vanishing conditions.
 	\par For $i=1$, and $v \nmid p$ these maps are the inclusions of $H^1_{ur}(E_v,T(\eta^{-1}))$ into their saturations, which are injective with finite cokernel. For the remaining case of $i=1$ and $v \mid p$, the result follows from \citep[Lemma 4.1.7]{FK06} since our choice of character $\eta$ ensures that the subrepresentation $V(\eta^{-1})^0$ associated to the rank 0 local condition in region $\Sigma^{(0)}$ in Figure 1 has exactly the positive Hodge numbers of $V(\eta^{-1})$, and thus $$\mathbb{D}_{\mathrm{dR}}(V(\eta^{-1})^0) \cong \mathbb{D}_{\mathrm{dR}}(V(\eta^{-1}))/\mathbb{D}^0_{\mathrm{dR}}(V(\eta^{-1})).$$ Thus the lemma tells us $H^1(E_v, V(\eta^{-1})^0)=H^1_f(E_v,V(\eta^{-1}))$, which gives us the result by \citep[Prop B.2.4]{Rub00}.
 
	\end{proof}
	
	\subsection{The case of inert $p$}

	\par So far in \S4 we have exclusively dealt with the rank 0 conjecture when $p$ splits in $K$, but as mentioned in previous remarks there are a few reasons we cannot repeat the arguments for inert $p$. Firstly, the correct rank 0 Panchishkin subrepresentation doesn't exist when $p$ is inert and correspondingly we don't have a Coleman map which will map our Euler system to a distribution. One possible approach is to use the theory of $L$-analytic characters to construct a regulator map when the twist $\eta$ satisfies some extra conditions, following ideas in \cite{SV15} where a lot of the hard work has already been done. We will need some additional constraints on our twist $\eta$, so that the middle piece of our representation looks like a power of a Lubin--Tate character (or more generally an $E_p$-analytic character). We can then obtain similar results to Theorem 4.3 but over an $E_p$-analytic distribution algebra which doesn't descend to an Iwasawa algebra in the usual sense. The module theory over this algebra would be more complicated and we would have to state our results more carefully. This is a work in progress by the author.
	\par Moreover, we are yet to connect our motivic $p$-adic $L$-function to an automorphic one interpolating complex $L$-values (after all this is the purpose of an Iwasawa main conjecture), which has been constructed in the split case in the work of Gyujin Oh \citep{Oh22} using a higher Hida theory for GU(2,1) constructed by Manh-Tu Nguyen \citep{Ngu20}. Future work will aim in reconciling these definitions via an explicit reciprocity law, and subsequently to do the same in the inert case despite the lack of $p$-adic analytic constructions or a higher Hida theory at the time of submission.
	
	\subsubsection*{Acknowledgements}
		I am grateful to my supervisor David Loeffler for presenting me with this problem and the many paths that have emerged from it, and for the helpful comments and guidance. I am also grateful for the useful and constructive conversations with Kazim Büyükboduk, Pak-Hin Lee, Rob Rockwood, Otmar Venjakob and many others.
	
	\bibliographystyle{amsalpha}
	\bibliography{bibliography}
	
\end{document}